\documentclass[12pt,oneside,reqno]{amsart}
\usepackage{amsmath,amsfonts,latexsym,graphicx,amssymb,url}
\usepackage{hyperref}
\usepackage{amsmath,txfonts,pifont,bbding,pxfonts}
\usepackage{ulem}
\usepackage{manfnt}
\usepackage{tabularx}
\usepackage[pagewise]{lineno}
\usepackage{booktabs}
\usepackage{caption}

\usepackage[active]{srcltx}
\usepackage{cases}
\usepackage{wasysym,pstricks, enumerate,subfigure}
\usepackage{lscape,color}
\usepackage{soul}
\usepackage[none]{hyphenat}
\usepackage[mathscr]{euscript}
\usepackage{wrapfig,subfigure,graphicx}
\usepackage{enumitem}

\usepackage{geometry}
\newcommand{\dist}{\text{dist}} 
\newcommand{\diam}{\text{{\rm diam}}}

\numberwithin{equation}{section}
\numberwithin{equation}{subsection}

\newtheorem{theorem}{Theorem}[section]

\newtheorem{lemma}[theorem]{Lemma}
\newtheorem{proposition}[theorem]{Proposition}
\newtheorem{definition}[theorem]{Definition}
\newtheorem{remark}[theorem]{Remark}

\numberwithin{equation}{section}
\numberwithin{equation}{subsection}

\begin{document}

\title[Fractional Infinity Laplacian with Obstacle]{Fractional Infinity Laplacian with Obstacle}
\author[Samer Dweik and Ahmad Sabra]{Samer Dweik and Ahmad Sabra\\}
\thanks{MSC Classification Codes: 35D40, 35J60, 35J65.}
\email{dweik@qu.edu.qa, asabra@aub.edu.lb}

\begin{abstract}

This paper deals with the obstacle problem for the fractional infinity Laplacian with nonhomogeneous term $f(u)$, where $f:\mathbb{R}^+ \mapsto \mathbb{R}^+$:
$$\begin{cases} L[u]=f(u) &\qquad in \,\,\,\,\, \{u>0\}\\
u \geq 0 &\qquad in \,\,\,\,\, \Omega\\
u=g &\qquad on\,\,\,\,\, \partial \Omega\end{cases},$$
with 
$$L[u](x)=\sup_{y\in \Omega,\,y\neq x}\dfrac{u(y)-u(x)}{|y-x|^{\alpha}}\,+\,\inf_{y\in \Omega,\,y\neq x} \dfrac{u(y)-u(x)}{|y-x|^\alpha},\qquad 0<\alpha<1.$$\\
Under the assumptions that $f$ is a continuous and monotone function and that the boundary datum $g$ is in $C^{0,\beta}(\partial\Omega)$ for some $0<\beta<\alpha$, we prove existence of a solution $u$ to this problem. Moreover, this solution $u$ is $\beta-$H\"olderian on $\overline{\Omega}$. Our proof is based on an approximation of $f$ by an appropriate sequence of functions $f_\varepsilon$ where we prove using Perron's method the existence of solutions $u_\varepsilon$, for every $\varepsilon>0$. Then, we show some uniform H\"older estimates on $u_\varepsilon$ that guarantee that $u_\varepsilon \rightarrow u$ where this limit function $u$ turns out to be a solution to our obstacle problem.
\end{abstract}

\keywords{Fractional Infinity Laplacian, Viscosity solutions, Non-local and Non-linear equations, Obstacle Problem}

\maketitle

\tableofcontents

\section{Introduction}
 The analysis of solutions to the infinity Laplacian equations dates back to the early results of {{Aronsson}} in \cite{Aronson1,Aronson2}. Let $\Omega$ be a Lipschitz domain in $\mathbb{R}^n$ and $g$ be a Lipschitz function on $\partial \Omega$. Then, the optimal Lipschitz extension $u$ of the boundary datum $g$ minimizing the $L^\infty-$norm of the gradient of $u$ on ${\Omega}$ (i.e. $||\nabla u||_{L^\infty(\Omega)}$) is a solution in the viscosity sense of the following Dirichlet infinity Laplacian boundary value problem: 
\begin{equation} \label{Infinity Laplacian}
\begin{cases}
\Delta_{\infty} u:= D^2 u \,\nabla u \cdot \nabla u =0\qquad &\text{in}\,\,\,\, \Omega\\
u=g & \text{on}\,\,\,\, \partial \Omega
\end{cases}.
\end{equation}
Generalization to the Aronsson Functional $||F(x,u,\nabla u)||_{L^\infty(\Omega)}$ has been also extensively studied in \cite{BarronBook,BarronJensen,Crandall}.

From \cite{Manfredi}, the solution $u$ to the infinity Laplacian problem \eqref{Infinity Laplacian} can also be obtained as the limit when $p\to \infty$ of the minimizers $u_p$ of the $p$-Laplacian mimimization problem $$\min\left\{\int_\Omega |\nabla u|^p\,:\,u \in W^{1,p}(\Omega),\, u=g\,\,\,\mbox{on}\,\,\,\partial\Omega\right\}.$$

On the other hand, the fractional Laplacian operator is a non-local operator which appears in many differential equations related to non-local tug-of-war game \cite{Cafarelli,Peres}, optimal control problems \cite{Antil}, image processing \cite{Aleotti}, SQG and porous medium models \cite{Abdo,Vazquez}.
In \cite{Chambole}, the authors studied the limit of the fractional $p-$Laplacian when $p \to \infty$. More precisely, they consider the minimization problem
\begin{equation} \label{eq:variational}
    \min\left\{\iint_{\Omega\times \Omega} \dfrac{|u(x)-u(y)|^p}{|x-y|^{\alpha p}}\,\mathrm{d}x\,\mathrm{d}y\,:\,u \in W^{s,p}(\Omega),\,u=g\,\,\,\mbox{on}\,\,\,\partial\Omega\right\},
\end{equation}
where 
\,$\alpha\in (0,1]$ is fixed, \,$s=\alpha-\dfrac{n}{p}$,\, $g\in C^{0,\alpha}(\partial \Omega)$\, and the fractional Sobolev space $W^{s,p}(\Omega)$ is defined as follows: 
$$W^{s,p}(\Omega)=\left\{u\in L^p(\Omega): \|u\|_{p}+[u]_{s,p,\Omega}<\infty\right\}$$ where 
    $$[u]_{s,p,\Omega}=\left(\iint_{\Omega\times \Omega}\dfrac{|u(x)-u(y)|^p}{|x-y|^{sp+n}}\, \mathrm{d} x\,\mathrm{d}y\right)^{1/p}.$$\\
   Let $u_p$ be the unique minimizer of Problem \eqref{eq:variational}. Then, it is easy to see that $u_p$ solves the following Euler Lagrange equation:
for any test function $\varphi\in C_0^{\infty}({\Omega})$, one has 
$$
\iint_{\Omega\times \Omega}\dfrac{|u(x)-u(y)|^{p-1}}{|x-y|^\alpha}\mbox{sgn}(u(x)-u(y))\,(\varphi(x)-\varphi(y))\,\mathrm{d}x\,\mathrm{d}y=0
$$
where $\mbox{sgn}(s)=\frac{s}{|s|}$ for $s \neq 0$.
It is then proved in \cite[Proposition 6.4]{Chambole} that 
$u_p$ is a viscosity solution of the equation:
\begin{equation} \label{L-p viscosity solution}
L_p[u]:=\int_{\Omega} \bigg|\dfrac{u(x)-u(y)}{|x-y|^{\alpha }}\bigg|^{p-1}\frac{\mbox{sgn}(u(x)-u(y))}{|x-y|^\alpha}\, \mathrm{d}y=0.
\end{equation}
From \cite[Theorem 1.1]{Chambole}, $u_p$ converges uniformly to a function $u_{\infty}\in C^{0,\alpha}(\overline \Omega)$ which is a viscosity solution to the H\"older (or fractional) infinity Laplace equation (we can see this operator $L$ as the limit of $L_p$ when $p \to \infty$): 
\begin{equation} \label{fractional infinity Lap}
L[u](x):=\sup_{y\in \Omega,\,y\neq x}\dfrac{u(y)-u(x)}{|y-x|^{\alpha}}+\inf_{y\in \Omega,\,y\neq x} \dfrac{u(y)-u(x)}{|y-x|^\alpha}=0.
\end{equation}
Moreover, $u_\infty$ is an optimal H\"older extension of the boundary datum $g \in C^{0,\alpha}(\partial\Omega)$, in the sense that the H\"older seminorm $[u_\infty]_{\alpha,\Omega}$ is always less than or equal $[u]_{\alpha,\Omega}$ for any $\alpha-$H\"older function $u$ such that $u=g$ on $\partial \Omega$, where
$$[u]_{\alpha,\Omega}=\sup_{x,y \in \Omega,\,x\neq y}\dfrac{|u(x)-u(y)|}{|x-y|^{\alpha}}.$$
In \cite{Merida}, the authors have considered the associated Dirichlet obstacle problem to \eqref{fractional infinity Lap}, i.e. they studied the fractional infinity Laplacian problem but in the presence of an obstacle $\psi$:
\begin{equation} \label{Frac inf with obs}
\begin{cases} L[u]=0 &\qquad in \,\,\,\,\,\{u>\psi\},\\
 L[u]\leq 0 &\qquad in \,\,\,\,\,\{u=\psi\},\\
  u \geq \psi &\qquad in \,\,\,\,\,\Omega,\\
 u=g &\qquad on\,\,\,\,\, \partial \Omega.\end{cases}
 \end{equation}
 \\
Following the approximation of \eqref{fractional infinity Lap} by the fractional $p-$Laplacian as in \cite[Section 6]{Chambole}, the authors in \cite{Merida} proved existence of a viscosity solution to \eqref{Frac inf with obs} by studying the limit when $p \to \infty$ of the following fractional $p-$Laplacian problem with obstacle:  
\begin{equation} \label{eq:variational obs}
    \min\bigg\{\iint_{\Omega\times \Omega} \dfrac{|u(x)-u(y)|^p}{|x-y|^{\alpha p}}\,\mathrm{d}x\,\mathrm{d}y\,:\,u \in W^{s,p}(\Omega),\,u \geq \psi\,\,\,\mbox{in}\,\,\,\Omega,\,u=g\,\,\,\mbox{on}\,\,\,\partial\Omega\bigg\}.
\end{equation}
On the other side, we note that the existence of a solution to the nonhomogeneous fractional infinity Laplacian, i.e. to equation \eqref{fractional infinity Lap} but with right hand term $f(x)$, cannot be obtained by means of a $p-$Laplacian approximation. However, the authors of \cite{Chambole} have also considered the nonhomogeneous version of \eqref{fractional infinity Lap}:
\begin{equation} \label{nonhomogeneous frac infinity Lap}
\begin{cases} L[u]=f(x) &\qquad \text{in} \,\,\,\,\,\Omega,\\ u=g &\qquad \text{on}\,\,\,\,\, \partial \Omega.\end{cases}
\end{equation}
In fact, they prove that if $f \in C(\Omega) \cap L^\infty(\Omega)$ and $g \in C(\partial\Omega)$, then a viscosity solution $u \in C(\overline{\Omega})$ to Problem \eqref{nonhomogeneous frac infinity Lap} exists. Moreover, they show that solutions $u$ are locally $\beta-$H\"older continuous, for any $0<\beta<\alpha$, and a global $\beta-$H\"older estimate was also obtained when $g\in C^{0,\beta}(\partial \Omega)$. In addition, there is a partial result in \cite{Chambole} about the uniqueness of the solution $u$ to \eqref{nonhomogeneous frac infinity Lap}. In the homogeneous case (i.e. when $f=0$), the solution $u$ is unique and locally Lipschitz (see \cite[Theorem 1.5]{Chambole}) and an implicit representation of this solution $u$ has been proven; 
$u(x)$ is the unique root \,$r$\, to the following equation: 
$$\sup_{y\in \partial\Omega}\,\dfrac{g(y)-r}{|y-x|^{\alpha}}\,+\,\inf_{y\in \partial\Omega} \dfrac{g(y)-r}{|y-x|^\alpha}=0. $$
\\
{ The uniqueness of the solution to $L[u]=f(x)$ when $f$ is signed, continuous and bounded on $\Omega$ is studied under some growth condition on the solution outside the domain $\Omega$ in \cite{Fejne}, except there the operator $L$ is slightly different where the supremum and the infimum are taken over the whole space $\mathbb{R}^n$}. The uniqueness in the general nonhomogeneous case Problem \eqref{fractional infinity Lap} is still widely an open question.  Moreover, the optimal $C^{0,\alpha}-$regularity of the solution remain open for general functions $f$.\\ 


Motivated by the results of \cite{Chambole}, we study in this paper the fractional infinity Laplacian equation but with nonhomogeneous term $f(u)$ that depends on the solution $u$. To be more precise, we aim to prove the
existence of a solution $u$ to the following equation that satisfies also the Dirichlet boundary condition $u=g$ on $\partial\Omega$:
\begin{equation} \label{Frac Inf Lap Eq}
L[u]=f(u) \qquad \mbox{in}\,\,\,\,\Omega.
\end{equation}
We note that the dependence of the right hand term $f(u)$ on the solution itself makes the problem more complicated. So, the question here is to find the good assumptions on $f$ that guarantee the existence of a solution to \eqref{Frac Inf Lap Eq}. Like in \cite{Chambole}, the continuity of $f$ will be  essential here too. But, we will not assume that $f$ is bounded (which is a required condition in \cite{Chambole}). However, we will impose a monotonicity condition on $f$ and prove by the mean of maximum principle that if $f$ is monotone and $g$ is $\beta-$H\"older continuous then a solution $u$ to \eqref{Frac Inf Lap Eq} exists satisfying $u=g$ on 
$\partial \Omega.$
Local and global H\"older regularity of solutions will be also proved.

In addition, we will consider equation \eqref{Frac Inf Lap Eq} but in the presence of an obstacle. Concretely, we will prove existence of a function $u$ that is nonnegative over $\Omega$ (here $u \geq 0$ represents the obstacle), that takes the datum $g$ on $\partial\Omega$, and solves the following equation \eqref{Frac Inf Lap Eq} but inside the positivity set $\{u>0\}$:  
\begin{equation} \label{Frac Inf Lap Eq Obstacle}
L[u]=f(u) \qquad \mbox{in}\,\,\,\,\{u>0\}.
\end{equation}

The paper is organized as follows.  In Section \ref{sec: Preliminary}, we show some properties on the operator $L$. In particular, we show that the function $|x-x_0|^{\beta}$ (where $\beta\leq \alpha$) is a strict subsolution to \eqref{fractional infinity Lap}; this will be fundamental in our later analysis. In section \ref{sec:viscosity}, we introduce the notion of viscosity (sub/super) solution to \eqref{Frac Inf Lap Eq} and show in Proposition \ref{comparison principle} the comparison principle. Moreover, we will prove a stability result on subsolutions. We also develop a Perron's Method argument in Section \ref{sec:Perron}
and prove the following existence and regularity results. 
\begin{theorem}\label{thm:Main 1}
    Assume $f:\mathbb{R} \mapsto \mathbb{R}$ is continuous and monotone non-decreasing, and the boundary datum $g$ is {{continuous}} on $\partial\Omega$. Then, the following fractional infinity Laplacian problem:
$$\begin{cases} L[u]=f(u) &\qquad in \,\,\,\,\,\Omega,\\ u=g &\qquad on\,\,\,\,\, \partial \Omega.\end{cases}$$
    has a solution $u$, {{which is locally $\beta-$H\"olderian for any $0<\beta <\alpha \leq 1$. Moreover, $u \in C^{0,\beta}(\overline{\Omega})$ as soon as $g \in C^{0,\beta}(\partial\Omega)$. In addition, $u$ is locally $\alpha-$H\"olderian if $f$ is nonnegative, for any $\alpha \in (0,1)$.}} 
\end{theorem}
We note that the solution constructed in the proof of Theorem \ref{thm:Main 1} is non-negative when both $f$ and $g$ are non-negative; this will allow us to introduce the obstacle problem in Section \ref{sec:obstacle} and show the following second main result of the paper.
\begin{theorem}\label{Th1.2}
    Assume \,$0<\alpha <1$, $f$ is nonnegative, continuous and monotone non-decreasing on $[0,\infty)$ and $g \in C^{0,\beta}(\partial\Omega)$ (for some $0<\beta<\alpha$) is nonnegative. Then, there exists a nonnegative $\beta-$H\"older solution $u$ to the following obstacle fractional infinity Laplacian problem:
$$\begin{cases} L[u]=f(u) &\qquad in \,\,\,\,\,\{u>0\},\\ u=g &\qquad on\,\,\,\,\, \partial \Omega.\end{cases}$$
{{Moreover, $u$ is locally $\alpha-$H\"older continuous on $\Omega$.}}
\end{theorem}
The main idea of the proof of Theorem \ref{Th1.2} is to approximate the function $f$ with a sequence of non-decreasing continuous functions and use the result of Section \ref{sec:viscosity} to obtain a sequence of solutions to \eqref{Frac Inf Lap Eq} converging to a solution for the obstacle problem \ref{Frac Inf Lap Eq Obstacle} with boundary data $g$.

\section{Preliminaries}\label{sec: Preliminary}
 In this section, we introduce some properties of the fractional infinity Laplacian operator $L$ that we will use later in our paper. First of all, we
define the following intermediary operators
$$L^+[u]=\sup_{y\in \overline{\Omega},\,y\neq x}  \dfrac{u(y)-u(x)}{|y-x|^{\alpha}}\qquad\qquad \mbox{and}\qquad\qquad L^-[u]=\inf_{y\in \overline{ \Omega},\,y\neq x} \dfrac{u(y)-u(x)}{|y-x|^{\alpha}}.$$\\
Recalling the definition of the operator $L$, we clearly have $L[u]=L^+[u]+L^-[u]$.

We start by the following simple lemma that we use frequently in the sequel (we give the proof just for the sake of completeness).
\begin{lemma}\label{metric} Fix $\alpha\in (0,1]$. Then, for all $x,\,y\in \mathbb R^n$, we have the $\alpha-$triangle inequality: 
$$|x+y|^{\alpha}\leq |x|^{\alpha}+|y|^{\alpha}.$$ 
In addition, the equality holds if and only if we either have $x=0$ or $y=0$.
\end{lemma}

\begin{proof}
Let $a,\,b>0$. For any $r\geq 0$, we define the function  $h(r)=(r+b)^\alpha-r^\alpha-b^\alpha.$ 
Notice that 
    $$h'(r)=\alpha\left[(r+b)^{\alpha-1}-r^{\alpha-1}\right]<0.$$
Hence, we infer that $h$ is strictly decreasing on $[0,\infty)$ and so, one has the following inequality:
\begin{equation} \label{ineQuality}
h(a)=(a+b)^{\alpha}-a^\alpha-b^{\alpha}<h(0)=0.
\end{equation}
For $x,\,y\in \mathbb R^n$ non zero, we get from \eqref{ineQuality} with $a=|x|$, $b=|y|$ and using the classical triangle inequality, that
$$|x+y|^{\alpha}\leq (|x|+|y|)^{\alpha}<|x|^{\alpha}+|y|^{\alpha}.$$
Finally, we note that equality follows immediately when $x=0$ or $y=0$. $\qedhere$

\end{proof}

Fix $x_0\in \overline{\Omega}$. Then, we define the barrier function $\psi_{\beta,x_0}(x)=|x-x_0|^\beta$. First, we calculate $L[\psi_{\beta,x_0}]$ when $0<\beta<\alpha$. We note that $\psi_{\beta,x_0}$ will be used later in Section \ref{sec:Perron} to construct sub/supersolutions as well as to show $\beta-$H\"older regularity on solutions.

\begin{proposition}\label{prop:alpha>beta}
Assume $0<\beta<\alpha\leq 1$, $x_0\in \overline{\Omega}$ and $\psi_{\beta,x_0}(x)=|x-x_0|^{\beta}$. Then, for every $x\in \Omega\setminus\{x_0\}$, we have
\begin{equation}\label{equ:estimate beta<alpha}
L [\psi_{\beta,x_0}](x)\leq |x-x_0|^{\beta-\alpha}\left(\dfrac{r_{\star}^{\beta}-1}{(r_{\star}-1)^\alpha}-1\right)<0,\end{equation}
where $r_{\ast}>\dfrac{1-\beta}{\alpha-\beta}$ is the unique solution in $(1,\infty)$ to the following equation:
$$(\alpha-\beta)\,r^{\beta}+\beta\, r^{\beta-1}-\alpha=0.$$
\end{proposition}
\begin{proof}
First, it is clear that
\begin{equation} \label{est on L-}
L^-[\psi_{\beta,x_0}](x)\leq \dfrac{\psi_{\beta,x_0}(x_0)-\psi_{\beta,x_0}(x)}{|x_0-x|^{\alpha}}=-|x-x_0|^{\beta-\alpha}.
\end{equation}
On the other hand,
 \begin{align*} L^+[\psi_{\beta,x_0}](x) 
  &=\sup_{y\in \overline{\Omega},\, y\neq x}\dfrac{|y-x_0|^{\beta}-|x-x_0|^{\beta}}{|y-x|^{\alpha}}=\sup_{y\in \overline{\Omega},\,|y-x_0|> |x-x_0|}\dfrac{|y-x_0|^{\beta}-|x-x_0|^{\beta}}{|y-x|^{\alpha}}\\
  &\leq \sup_{y\in \overline{\Omega},\, |y-x_0|> |x-x_0|}\dfrac{|y-x_0|^{\beta}-|x-x_0|^{\beta}}{(|y-x_0|-|x-x_0|)^{\alpha}}= |x-x_0|^{\beta-\alpha}\sup_{y\in \overline{\Omega},\, |y-x_0|> |x-x_0|}\dfrac{\left(\dfrac{|y-x_0|}{|x-x_0|}\right)^{\beta}-1}{\left(\dfrac{|y-x_0|}{|x-x_0|}-1\right)^{\alpha}}.
  \end{align*}
  Hence
\begin{equation} \label{est on L+}
  L^+[\psi_{\beta,x_0}](x)\leq |x-x_0|^{\beta-\alpha}\sup_{1<r<\frac{\diam(\Omega)}{|x-x_0|}} \Psi(r),
  \end{equation}
where 
$\Psi(r):=\dfrac{r^{\beta}-1}{(r-1)^{\alpha}}$. We note that $\lim_{r\to 1^+} \Psi(r)=\begin{cases} 0 \quad & \,\,\mbox{if}\,\,\,\alpha<1\\ \beta & \,\,\mbox{if}\,\,\, \alpha=1\end{cases}$ and $\lim_{r\to \infty}\Psi(r)=0$. Moreover, one has
$$\Psi'(r)=\dfrac{\beta r^{\beta-1}(r-1)^{\alpha}-\alpha(r-1)^{\alpha-1}(r^{\beta}-1)}{(r-1)^{2\alpha}}=\dfrac{\beta r^{\beta-1}(r-1)-\alpha (r^{\beta}-1)}{(r-1)^{\alpha+1}}=\dfrac{(\beta-\alpha)r^{\beta}-\beta r^{\beta-1}+\alpha}{(r-1)^{\alpha+1}}.$$
Now, set \,$p(r)=(\beta-\alpha)r^{\beta}-\beta r^{\beta-1}+\alpha.$ 
Notice that $p(1)=0$,\, $\lim_{r\to \infty} p(r)=-\infty$, and we have
$$p'(r)=\beta(\beta-\alpha)r^{\beta-1}-\beta(\beta-1)r^{\beta-2}=\beta r^{\beta-2}[(\beta-\alpha)r-(\beta-1)].$$
Let $r_0=\dfrac{1-\beta}{\alpha-\beta}$ be the unique root of $p^\prime(r)=0$. From above we deduce that $p$ has a unique root $r_\star >r_0$ such that $$\sup_{r>1}\Psi(r)=\Psi(r_{\star}).$$
Combining the estimates \eqref{est on L-} on $L^-$ and \eqref{est on L+} on $L^+$ , we conclude \eqref{equ:estimate beta<alpha}. But, from Lemma \ref{metric}, we have
$$r_\star^\beta< (r_{\star}-1)^{\beta}+1\leq (r_{\star}-1)^{\alpha}+1.$$
Hence, 
we have \,$L[\psi_{\beta,x_0}](x)<0$. $\qedhere$
\end{proof}

We now give an estimate on $L [\psi_{\beta,x_0}]$ but in the case when $\beta=\alpha$. This will be used in Section \ref{sec:obstacle} to show $\alpha-$H\"older regularity on solutions to the obstacle problem \eqref{Frac Inf Lap Eq Obstacle}.
\begin{proposition}\label{alpha=beta}
Letting $\psi_{\alpha,x_0}(x)=|x-x_0|^{\alpha}$ with $\alpha\in (0,1)$ and $x_0\in \overline{\Omega}$. Then, one has
$$L[\psi_{\alpha.x_0}](x) \leq -1+\dfrac{ \left(\dfrac{\diam(\Omega)}{|x-x_0|}\right)^\alpha-1}{\left(\dfrac{\diam (\Omega)}{|x-x_0|}-1\right)^{\alpha}}<0,\qquad \mbox{for all}\,\,\,x \neq x_0.$$
\end{proposition}
\begin{proof}

From Lemma \ref{metric}, one has $|x-x_0|^{\alpha}\leq |x-y|^{\alpha}+|y-x_0|^{\alpha}$ and so for $y\neq x$, we have the following:
$$\dfrac{|y-x_0|^{\alpha}-|x-x_0|^{\alpha}}{|y-x|^{\alpha}}\geq -1,$$
with equality attained at $y=x_0$. So,
$L^-[\psi_{\alpha,x_0}](x)=-1$. Proceeding as in Proposition \ref{prop:alpha>beta}, one has 
    $$L^+[\psi_{\alpha,x_0}](x)\leq \sup_{1<r<\frac{\diam(\Omega)}{|x-x_0|}}\Psi(r),$$
    with $\Psi(r)=\dfrac{r^{\alpha}-1}{(r-1)^{\alpha}}.$ In this case,
    $\Psi'(r)=\dfrac{\alpha(1-r^{\alpha-1})}{(r-1)^{\alpha+1}}>  0.$
    Consequently, we get that
    $$L^+[\psi_{\alpha,x_0}](x)\leq \Psi\left(\dfrac{\diam(\Omega)}{|x-x_0|}\right)=\dfrac{ \left(\dfrac{\diam (\Omega)}{|x-x_0|}\right)^\alpha-1}{\left(\dfrac{\diam(\Omega)}{|x-x_0|}-1\right)^{\alpha}}<\lim_{r\to \infty}\psi(r)=1.$$
If \,$\alpha<1$, then we have for $x\neq x_0$
$$L[\psi_{\alpha,x_0}](x) \leq -1+\dfrac{ \left(\dfrac{\diam(\Omega)}{|x-x_0|}\right)^\alpha-1}{\left(\dfrac{\diam (\Omega)}{|x-x_0|}-1\right)^{\alpha}}<0. \qedhere$$ 
\end{proof}
  In the following lemma, we will show some estimates on $L^\pm[\varphi]$ in the case when $\varphi$ is a smooth function.

\begin{lemma}\label{sign L+}
Assume $\varphi$ is a $C^1$ function in a neighborhood of some point $x_0 \in \Omega$. Then, for $\alpha\in (0,1]$, we have
$$L^-[\varphi](x_0) \leq 0 \leq L^+[\varphi](x_0).$$
Moreover, if $\alpha=1$ then
$$ L^+[\varphi](x_0)\geq |\nabla \varphi(x_0)| \qquad \mbox{and}\qquad  L^-[\varphi](x_0)\leq -|\nabla \varphi(x_0)|.$$
\end{lemma}
\begin{proof}
Let \,${\bf e}$\, be a unit vector in $\mathbb R^n$. From the definition of $L^+$, one has the following:
$$L^+[\varphi](x_0)\geq \lim_{h\to 0} \dfrac{\varphi(x_0+h {\bf e})-\varphi(x_0)}{|h|^{\alpha}}=\lim_{h\to 0} \dfrac{\varphi(x_0+h {\bf e})-\varphi(x_0)}{|h|}|h|^{1-\alpha}
$$
$$=\begin{cases} 0  & \,\,\,\mbox{if}\,\,\,\,\, 0<\alpha<1,\\
\nabla\varphi(x_0)\cdot {\bf e} &\,\,\, \mbox{if}\,\,\,\,\, \alpha =1.\end{cases}
$$
For $\alpha=1$, taking ${\bf e}=\dfrac{\nabla \varphi(x_0)}{|\nabla \varphi(x_0)|}$ when $\nabla \varphi(x_0)\neq 0$, we deduce in this case that 
$L^+[\varphi](x_0)\geq |\nabla \varphi(x_0)|$. 

The estimates on $L^-[\varphi]$ follow directly from the fact that $L^-[\varphi]=-L^+[-\varphi]$. $\qedhere$
\end{proof}

Next, we show that $L^\pm[\varphi]$ must be continuous for smooth functions $\varphi$.

\begin{proposition}\label{lem:continuity}
If $\varphi\in C^1(\Omega)$, then $L^{\pm}[\varphi]\in C(\Omega)$.
\end{proposition}

\begin{proof}

Fix $x_0\in \Omega$ and let $\{x_n\}$ be a sequence of points converging to $x_0$. We show that 
$$L^+[\varphi](x_n) \rightarrow L^+[\varphi](x_0).$$
We have 
$$L^+[\varphi](x_n)=\sup_{y \in \bar{\Omega},\,y\neq x_n} \frac{\varphi(y)-\varphi(x_n)}{|y-x_n|^\alpha}.$$
First, assume that there exists an $\varepsilon_0>0$ such that for all $n$ there is a point $y_n \in \bar{\Omega}\backslash B(x_0,\varepsilon_0)$ satisfying
$$L^+[\varphi](x_n)= \frac{\varphi(y_n)-\varphi(x_n)}{|y_n-x_n|^\alpha} \geq \frac{\varphi(y)-\varphi(x_n)}{|y-x_n|^\alpha},\qquad \mbox{for all}\,\,\,\,y \in \bar{\Omega},\,y \neq x_n.$$
Hence, $\liminf_{n\to \infty} L^+[\varphi](x_n)\geq \dfrac{\varphi(y)-\varphi(x_0)}{|y-x_0|^{\alpha}}$ for every $y \neq x_0$ and so, $\liminf_{n\to \infty} L^+[\varphi](x_n) \geq L^+[\varphi](x_0).$
On the other hand, $y_n$ has a convergent subsequence $y_{n_k}$ say to $y_0$, then since $y_0\neq x_0$,
$$\lim_{k\to \infty} L^+[\varphi](x_{n_k})=\dfrac{\varphi(y_0)-\varphi(x_0)}{|y_0-x_0|^{\alpha}}\geq \dfrac{\varphi(y)-\varphi(x_0)}{|y-x_0|^{\alpha}}\qquad \text{for all $y\in \bar \Omega, y\neq x_0$},$$
 and so 
$\lim_{k\to \infty} L^+[\varphi](x_{n_k})=L^+[\varphi](x_0).$
We conclude that in this case $\lim_{n\to \infty} L^+[\varphi](x_n)=L^+[\varphi](x_0).$

Now, assume that for every $n$ there is a point $y_n \neq x_n$ such that $|y_n-x_0| \rightarrow 0$ when $n \to \infty$
and
\begin{equation}\label{eq:ineq}
\frac{\varphi(y)-\varphi(x_n)}{|y-x_n|^\alpha} -\frac{1}{n} \leq L^+[\varphi](x_n) -\frac{1}{n} \leq \frac{\varphi(y_n)-\varphi(x_n)}{|y_n-x_n|^\alpha} 
\end{equation}
for all $y \in \bar{\Omega}$, $y \neq x_n$. Take $\delta>0$ such that $\overline{ B(x_0,\delta)}\subseteq \Omega$. Since $\varphi\in C^1(\Omega),$
then we clearly have
$$|\varphi(x)-\varphi(x')|\leq M|x-x'| \qquad\forall x,x'\in B(x_0,\delta).$$
If $\alpha<1$, for $n$ large, we get
$$\dfrac{|\varphi(y_n)-\varphi(x_n)|}{|y_n-x_n|^{\alpha}}\leq M|x_n-y_n|^{1-\alpha}\to 0.$$
Hence, \eqref{eq:ineq} becomes
$$\dfrac{\varphi(y)-\varphi(x_0)}{|y-x_0|^{\alpha}}\leq \limsup_{n\to \infty} L^+[\varphi](x_n)\leq 0,\qquad \mbox{for all}\,\, y\neq x_0.$$
Since $y$ is arbitrary, then $L^+[\varphi](x_0)\leq \limsup_{n\to \infty} L^+[\varphi](x_n)\leq 0.$ From Lemma \ref{sign L+}, we infer that
$$\lim_{n\to\infty} L^+[\varphi](x_n)=L^+[\varphi](x_0)=0.$$
Finally, we assume $\alpha=1$.  Notice that \eqref{eq:ineq} and Lemma \ref{sign L+} imply together that
$$|\nabla \varphi(x_0)|\leq L^+[\varphi](x_0)\leq \liminf_{n\to \infty} L^+[\varphi](x_n).$$
From the mean value theorem, there exists a point $\xi_n$ on the line segment joining $x_n$ to $y_n$ such that
$$\dfrac{\varphi(y_n)-\varphi(x_n)}{|y_n-x_n|}=\nabla \varphi(\xi_n)\cdot \dfrac{y_n-x_n}{|y_n-x_n|}\leq |\nabla \varphi(\xi_n)|.$$
Then, again from \eqref{eq:ineq},
$$\limsup_{n\to \infty} L^+[\varphi](x_n)\leq |\nabla \varphi(x_0)|,$$
concluding in this case that
$$\lim_{n\to \infty} L^+[\varphi](x_n)=L^+[\varphi](x_0)=|\nabla \varphi(x_0)|. \qedhere$$




\end{proof}
\begin{remark}
 Notice that the result of Proposition \ref{lem:continuity} fails if $\varphi$ is assumed to be only continuous. In fact, let $x_0\in \Omega$ and consider $\psi_{x_0}(x)=|x-x_0|.$ We have from the proof of Proposition \ref{alpha=beta} that $L^-[\psi_{x_0}](x)=-1$ for $x\neq x_0$ though $L^{-}[\psi_{x_0}](x_0)=1$.
\end{remark}

We complete this section with the following Lemma.
\begin{lemma}\label{lem:delta perturbation}
Assume 
$\varphi\in C^1(\Omega)\cap C
(\overline{\Omega})$ and $x_0\in \overline{\Omega}$. 
Define $\varphi_{\delta}(x)=\varphi(x)+\delta|x-x_0|^2$, with $\delta\in \mathbb R$. Then, 
we have
$$|L[\varphi_{\delta}](x)-L[\varphi](x)|\leq 4|\delta| \diam(\Omega)^{2-\alpha}.$$
In particular, this estimate implies that $L[\varphi_{\delta}]$ converges uniformly in $\delta$ to $L[\varphi]$.
\end{lemma}

\begin{proof}

Notice that for $y\neq x$, one has
\begin{align*}
    \dfrac{\varphi_{\delta}(y)-\varphi_{\delta}(x)}{|y-x|^{\alpha}}&=\dfrac{\varphi(y)-\varphi(x)}{|y-x|^{\alpha}}+\delta\,\dfrac{ |y-x_0|^2-|x-x_0|^2}{|y-x|^{\alpha}}\\
    &=\dfrac{\varphi(y)-\varphi(x)}{|y-x|^{\alpha}}+\delta\,\dfrac{(y-x)\cdot (y+x-2x_0)}{|y-x|^{\alpha}}.
\end{align*}
Hence,
$$|L^{\pm} [\varphi_{\delta}](x)-L^{\pm}[\varphi](x)|\leq |\delta| |y-x|^{1-\alpha}\left(|y-x_0|+|x-x_0|\right)\leq 2|\delta| \diam(\Omega)^{2-\alpha}.\qedhere$$
\end{proof}

\section{Existence of viscosity solution}\label{sec:viscosity}
In this section, we show the existence of a viscosity solution to \eqref{Frac Inf Lap Eq} by using the Perron’s method with some conditions on the function $f$.
\subsection{Subsolutions and Supersolutions}\label{subsec:Subsuper} First of all, we start by introducing the notions of viscosity subsolutions, supersolutions and solutions. For the theory of viscosity solutions, we refer the reader to \cite{Lions}.
\begin{definition}
Let $\Omega$ be an open bounded domain, $\alpha\in (0,1]$, 
and $f:\mathbb R\mapsto \mathbb R$. 
We say that 
$u:\overline{\Omega}\mapsto \mathbb R$ is a subsolution 
(resp. supersolution) to the equation $L[u]=f(u)$ and write 
$L[u]\geq f(u)$ (resp. $L[u]\leq f(u)$) if and only if 
$u:\overline{\Omega}\mapsto \mathbb R$ is upper semi-
continuous (resp. lower semi-continuous), and for any test 
function $\varphi\in C^1(\Omega)\cap C(\overline{\Omega})$ 
such that $u\leq \varphi$ (resp. $u\geq \varphi$) with 
equality at some $x_0\in \Omega$ then $-L[\varphi]
(x_0)+f(\varphi(x_0))\leq 0$ (resp. $-L[\varphi]
(x_0)+f(\varphi(x_0))\geq 0$). If the last inequality is 
strict for every such $\varphi$ and $x_0$ we say that $u$ 
is a strict subsolution (resp. supersolution) and write 
$L[u]>f(u)$ (resp. $L[u]<f(u)$).

We say that $u$ is a viscosity solution to $L[u]=f(u)$ if it is a viscosity subsolution and a viscosity supersolution to the same equation.
\end{definition}

\begin{remark}\label{rmk:super}\rm
Notice that since $L[-u]=-L[u]$ so if $u$ is a supersolution to $L[u]=f(u)$ then $-u$ is a subsolution to $L[v]=-f(-v)$; this follows from the fact that
$-u$ is upper semi-continuous and if $\varphi\in C^1(\Omega)\cap C(\overline\Omega)$ is such that $-u\leq \varphi$ with equality at $x_0$ then $u-(-\varphi)$ attains a minimum at $x_0$ and since $u$ is a supersolution to $L[u]=f(u)$, we get
$$-L[-\varphi](x_0)+f(-\varphi(x_0))\geq 0.$$
Yet, this implies that
$-L[\varphi](x_0)-f(-\varphi(x_0))\leq 0.$
\end{remark}
\begin{remark}\rm
    The notion of viscosity solution in this paper is stronger than the one in \cite{Chambole} where a viscosity solution there is not necessarily continuous but the upper semicontinuous envelope is a subsolution and the lower semicontinuous envelope is a supersolution.  
\end{remark}


Let $u$ be a viscosity solution on   $\Omega$. Since $L$ is a non-local operator, then it is not clear whether or not $u$ will be always a solution on a subset of $\Omega$. In the following proposition, we will show that this is true provided we remove only one point.

\begin{proposition} \label{subsolution on domain minus point}
Fix \,$x_0 \in \Omega$. Assume that \,$u$\, is a subsolution of $L[u]=f(u)$ on \,$\Omega$, then \,$u$\, is also a subsolution on \,$\Omega\setminus\{x_0\}$.
\end{proposition}

\begin{proof}
Let $\varphi\in C^1(\Omega\setminus \{x_0\})\cap C(\overline \Omega)$ such that $u\leq \varphi$ on $\overline{\Omega}$ with equality at some $x_1\in \Omega\setminus \{x_0\}$.
From Lemma \ref{lem:delta perturbation}, for $\delta>0$ $\varphi_{\delta}(x)=\varphi(x)+\delta |x-x_1|^2\in C^1(\Omega\setminus \{x_0\})\cap C(\overline \Omega)$ such that $u< \varphi_{\delta}$ for every $x\in \overline{\Omega}\setminus \{x_1\}$ with equality at $x_1$, and
$L[\varphi_{\delta}](x_1)\to L[\varphi](x_1)$ as $\delta\to 0$.

Fix $\delta>0$. We have $x_0\neq x_1$, let $\varepsilon_0>0$ be such that $\overline{B(x_0,\varepsilon_0)}\subseteq \Omega$ and not containing $x_1$. We construct a sequence $\varphi_n\in C^1(\Omega)$ converging uniformly to $\varphi_{\delta}$ in $\overline{B(x_0,\varepsilon_0)}$ and such that $\varphi_n= \varphi_{\delta}$ on $\Omega\setminus B(x_0,\varepsilon_0)$. We have $u<\varphi_{\delta}$ in $\overline{B(x_0,\varepsilon_0)}$ then for $n$ sufficiently large $u<\varphi_n$ in $\overline{B(x_0,\varepsilon_0)}$ and so $u<\varphi_n$ in $\overline{\Omega}\setminus\{x_1\}$ with equality at $x_1$. Since $u$ is a subsolution on $\Omega$, then
$$-L[\varphi_n](x_1)+f(\varphi_n(x_1))\leq 0.$$
But, we have that outside $\overline{B(x_0,\varepsilon_0)}$, $\varphi_n=\varphi_{\delta}$ and so, $\varphi_n(x_1)=\varphi_\delta(x_1)=\varphi(x_1)$. Therefore, one has
$$\sup_{y\in \overline{\Omega}\setminus \overline{B(x_0,\varepsilon_0)},\,y\neq x_1} \dfrac{\varphi_n(y)-\varphi_n(x_1)}{|y-x_1|^{\alpha}}=\sup_{y\in \overline{\Omega}\setminus \overline{B(x_0,\varepsilon_0)},\,y\neq x_1} \dfrac{\varphi_{\delta}(y)-\varphi_{\delta}(x_1)}{|y-x_1|^{\alpha}}.$$ 
Now, by uniform convergence of $\varphi_n$ and since $x_1 \notin \overline{B(x_0,\varepsilon_0)}$, then we have the following:
$$\lim_{n\to \infty} \sup_{y\in \overline{B(x_0,\varepsilon_0)}} \dfrac{\varphi_n(y)-\varphi_n(x_1)}{|y-x_1|^{\alpha}}= \sup_{y\in \overline{B(x_0,\varepsilon_0)}} \dfrac{\varphi_{\delta}(y)-\varphi_{\delta}(x_1)}{|y-x_1|^{\alpha}},$$
and similarly for the infimum.
Hence, we get that $\lim_{n\to \infty} L[\varphi_n](x_1)=L[\varphi_{\delta}](x_1),$
and so 
$$-L[\varphi_\delta](x_1)+f(\varphi_\delta(x_1)) \leq 0.$$
But, $\delta>0$ is arbitrary and $\varphi_\delta(x_1)=\varphi(x_1)$ so letting $\delta\to 0^+$, we infer that $-L[\varphi](x_1)+f(\varphi(x_1))\leq 0$, concluding that $u$ is a subsolution on $\Omega\setminus \{x_0\}$. $\qedhere$
\end{proof}

We next show a comparison principle when $f$ is non-decreasing which will help later in proving our H\"older estimates. 
\begin{proposition} \label{comparison principle}
Assume that $f$ is non-decreasing.
    Let $u$ be a subsolution (resp. supersolution) of $L[u]=f(u)$ and $v$ be a strict supersolution (resp. subsolution) such that $u \leq v$\, (resp. $u \geq v$) on \,$\partial \Omega$\, and \,$v \in C^1(\Omega) \cap C(\overline{\Omega})$. Then, $u < v$\, (resp. $u>v$) in $\Omega$.  
\end{proposition}
\begin{proof}
    Assume this is not the case, i.e. there is a point $x^\star \in \Omega$ such that $u(x^\star)- v(x^\star)=\max\limits_{x \in \Omega} [u(x)-v(x)]:=M\geq 0$. Note that the maximum is attained since $u$ is upper semicontinuous and $v$ is continuous on $\overline \Omega$. We clearly have $u \leq v + M$\, on \,$\overline{\Omega}$ with $u(x^\star)=v(x^\star)+M$. Since $u$ is a subsolution and $v \in C^1(\Omega) \cap C(\overline{\Omega})$, then we must have 
    $$-L[v+M](x^\star) +f(v(x^\star)+M) \leq 0.$$
    Yet, \,$f$\, is non-decreasing. Hence, we get that
     $$-L[v](x^\star) +f(v(x^\star)) \leq 0.$$
     But, this contradicts the fact that $v$ is a strict supersolution which concludes the proof.
\end{proof}
Now, we prove the following stability result of subsolutions when $f$ is continuous.

\begin{proposition} \label{stability}
   Assume $f$ is continuous. Let $\mathcal F$ be a non-empty family of subsolutions to \eqref{Frac Inf Lap Eq}. 
    Define $v(x):=\sup_{u\in \mathcal F} u(x)<\infty$ and assume that $v$ is { upper semi-continuous on $\overline{\Omega}$}.
    Then, 
    $v$ is a subsolution to \eqref{Frac Inf Lap Eq}.
\end{proposition}
\begin{proof}
We show that $-L[v]+f(v)\leq 0$ in the viscosity sense. We proceed by contradiction. Assume there exists $\varphi\in C^1(\Omega)\cap C(\overline \Omega)$ such that $v(x)-\varphi(x)\leq 0$ on $\overline \Omega$ with equality at $x_0$ and such that $-L[\varphi](x_0)+f(\varphi(x_0))>0.$ 
Take $\varphi_{\delta}(x)=\varphi(x)+\delta|x-x_0|^2$, with $\delta>0$, a peturbation of $\varphi$. We have \,$v(x)\leq \varphi(x)<\varphi_{\delta}(x)$ for $x\neq x_0$ and $v(x_0)=\varphi(x_0)=\varphi_{\delta}(x_0)$. Hence, $x_0$ is the unique maximum to $v-\varphi_{\delta}$. From Lemma \eqref{lem:delta perturbation}, we have
$$-L[\varphi_{\delta}](x_0)+f(\varphi_{\delta}(x_0)) = -L[\varphi_{\delta}](x_0)+f(\varphi(x_0))\geq -L[\varphi](x_0)+f(\varphi(x_0))-4\delta\diam(\Omega)^{2-\alpha}>0,$$
for \,$\delta$\, small enough.

Since $v(x_0)=\sup_{u\in \mathcal F}u(x_0)$, then for every $n \in \mathbb{N}^\star$ there exists $u_n\in\mathcal F$ such that
\begin{equation}\label{ineq: unxn}
{v(x_0)-\dfrac{1}{n}   \leq u_n(x_0).}
\end{equation}
Let $M_n=\sup_{x\in \overline \Omega}[u_n(x)-\varphi_\delta(x)]$, which by upper semi-continuity of $u_n$ and compactness of $\overline \Omega$ is attained at some $y_n\in \overline \Omega$. We have
$$M_n\to 0 \qquad \mbox{and}\qquad y_n\to x_0 \qquad \mbox{as}\qquad n\to \infty.$$
{
Indeed, we clearly have that
$u_n\leq v\leq \varphi_{\delta}$ on $\overline{\Omega}$ and so, $M_n\leq 0$.
On the other hand, by \eqref{ineq: unxn} and since $v(x_0)=\varphi_\delta(x_0)$, one has 
$$M_n\geq u_n(x_0)-\varphi_{\delta}(x_0)\geq -\dfrac{1}{n}.$$
Letting $n\to \infty$, we get that $\lim_{n\to \infty} M_n=0$. To show that $y_n\to x_0$, we notice that $$-\dfrac{1}{n}\leq M_n=u_n(y_n)-\varphi_{\delta}(y_n)\leq v(y_n)-\varphi(y_n)-\delta|y_n-x_0|^2\leq -\delta |y_n-x_0|^2.$$
 Hence, $|y_n-x_0|\leq \dfrac{1}{\sqrt{\delta n}} \rightarrow 0$.
}

Now, we complete the proof of the proposition. We define $\varphi_n=\varphi_{\delta}+M_n$. Notice that $$u_n=(u_n-\varphi_{\delta})+\varphi_{\delta}\leq M_n+\varphi_{\delta}=\varphi_n,$$
with equality at $y_n$. Then, since $u_n\in \mathcal F$, we get
$$0\geq -L[\varphi_n](y_n)+f(\varphi_n(y_n))\geq -L[M_n+\varphi_{\delta}](y_n)+f(M_n+\varphi_{\delta}(y_n))=-L[\varphi_{\delta}](y_n)+f(M_n+\varphi_{\delta}(y_n)).$$
Hence, thanks to Lemma \eqref{lem:continuity}, and the continuity of $f$, we conclude that
$-L[\varphi_{\delta}](x_0)+f(\varphi_{\delta}(x_0))\leq 0$,
a contradiction.
\end{proof}

\begin{remark}\label{rmk:stability super}\rm
 From Remark \ref{rmk:super}, we obtain a similar stability result for supersolutions, that is, if $\mathcal G$ is a family of supersolutions to $L[u]=f(u)$ and if $w(x):=\inf_{u\in \mathcal G}u(x)$ is { lower semi-continuous in $\overline{\Omega}$}, then $w$ is also a supersolution.
\end{remark}


\subsection{Perron's Method}\label{sec:Perron} The aim of this subsection is to construct a viscosity solution to \eqref{Frac Inf Lap Eq} by applying the Perron's method. First, we start by constructing a sub/supersolution.
\begin{lemma} \label{Existence of subsupersolutions}
Assume $f$ is non-decreasing and continuous, and {{$g$ is continuous on $\partial\Omega$.}}
Then, there exist a subsolution $u^-$ and a supersolution $u^+$ to \eqref{Frac Inf Lap Eq} such that $u^- \leq u^+$ on \,$\overline{\Omega}$\, and \,$u^-=u^+=g$\, on \,$\partial\Omega$.
\end{lemma}

{ 
\begin{proof}
Fix $\beta<\alpha$. For $x_0\in \partial \Omega$, $a\in \mathbb R$ and $b\geq 0$, we define the function $\phi_{x_0,a,b}$ on $\overline{\Omega}$ as follows:
$$\phi_{x_0,a,b}(x)=a-b|x-x_0|^{\beta}.$$
Recalling Proposition \ref{prop:alpha>beta}, we have  
\begin{equation}\label{ineq:Lphi}
L [\phi_{x_0,a,b}](x)=-b\, L[|x-x_0|^{\beta}]\geq -b\,|x-x_0|^{\beta-\alpha}\left(\dfrac{r_{\star}^{\beta}-1}{(r_{\star}-1)^\alpha}-1\right) \geq -b \,\diam(\Omega)^{\beta-\alpha}\left(\dfrac{r_{\star}^{\beta}-1}{(r_{\star}-1)^\alpha}-1\right),
\end{equation}
for all $x\in \Omega$. We take the set 
$$S=\bigg\{(x_0,a,b)\in \partial \Omega\times \mathbb R\times [0,\infty):  L[\phi_{x_0,a,b}(x)]\geq f(\phi_{x_0,a,b}(x))\,\,\,\,\text{in\,\, $\Omega$},\,\,{\phi_{x_0,a,b}}\leq g\,\,\,\mbox{on}\,\,\,\partial\Omega\bigg\}.$$
Notice that $S\neq \emptyset$. Indeed, using the fact that $f$ is non-decreasing and inequality \ref{ineq:Lphi}, $(x_0,a,b)\in S$ as soon as $x_0\in \partial \Omega$, \,$a\leq\min g$, and 
\begin{equation}\label{eq:b large}
b\geq \frac{\mbox{diam}(\Omega)^{\alpha-\beta}\,f(a)}{\bigg[1-\dfrac{r_{\star}^{\beta}-1}{(r_{\star}-1)^\alpha}\bigg]}.
\end{equation}
Notice also that for every $(x_0,a,b)\in S$, we have $a=\phi_{x_0,a,b}(x_0)\leq g(x_0)\leq \max g$ and so, $\phi_{x_0,a,b}(x)\leq \max g$, for all $x\in \overline{\Omega}.$
We then define $$u^-(x)=\sup_{(x_0,a,b)\in S}\phi_{x_0,a,b}(x).$$ We clearly have $u^-(x)<\infty$. Now, we show that $u^-=g$ on $\partial \Omega$. By definition of $S$, we have that $\phi_{x_0,a,b}\leq g$ on $\partial \Omega$ for every $(x_0,a,b)\in S$ and so, $u^-=\sup_{(x_0,a,b)\in S}\phi_{x_0,a,b} \leq g$ on $\partial \Omega$.
 
Given $\varepsilon>0$, by uniform continuity of $g$ there exists $r>0$ such that for all $x,y\in \partial \Omega$ with $|x-y|<r$, one has  
$$|g(x)-g(y)|<\epsilon.$$
 Fix $x_0\in \partial \Omega$. Take $a_{\varepsilon}=g(x_0)-\varepsilon$\, and \,$b_\epsilon\geq\frac{\mbox{diam}(\Omega)^{\alpha-\beta}\,f(a_\epsilon)}{\bigg[1-\dfrac{r_{\star}^{\beta}-1}{(r_{\star}-1)^\alpha}\bigg]}$. So, we have $L[\phi_{x_0,a_\epsilon,b_\epsilon}]\geq f(\phi_{x_0,a_\epsilon,b_\epsilon})$. Moreover, assume that 
$$b_{\epsilon}\geq \dfrac{g(x_0)-\epsilon -\min g}{r^{\beta}}.$$
Then, one has
$$\phi_{x_0,a_{\varepsilon},b_{\varepsilon}}(x)=g(x_0)-\varepsilon - b_\varepsilon |x-x_0|^\beta \leq g(x) \qquad\mbox{for all}\,\,\,x \in \partial\Omega.$$
Hence, $(x_0,a_{\epsilon},b_{\epsilon})\in S.$ In particular, we deduce that
$$u^-(x_0)\geq  \phi_{x_0,a_{\varepsilon},b_{\varepsilon}}(x_0)=g(x_0)-\varepsilon.$$
Letting $\varepsilon\to 0$, we get that $u^{-}(x_0)\geq g(x_0)$, concluding that $u^-(x_0)=g(x_0),$ for every $x_0\in \partial \Omega$.

We next prove that $u^-\in C(\overline{\Omega})$. Take $y_0\in \overline{\Omega}$ and sequence $y_n\in \overline \Omega$ that converges to $y_0$. Since
$u^-$ is the supremum of continuous functions, then it is lower semi-continuous and so,
$$u^-(y_0)\leq \liminf_{n\to \infty} u^-(y_n).$$
Let $(x_n,a_n,b_n)\in S$ be such that
\begin{equation}\label{phin}
    u^-(y_n)-\dfrac{1}{n}\leq \phi_{x_n,a_n,b_n}(y_n)\leq u^-(y_n).
\end{equation}
Assume $b_n$ has no unbounded subsequence that is $\lim_{n\to \infty}b_n=\infty$. Since $a_n \leq g(x_n)$, we have
$$u^-(y_n)-\dfrac{1}{n}\leq a_n-b_n|y_n-x_n|^{\beta}\leq \max g-b_n|y_n-x_n|^{\beta}.$$
So, $$|y_n-x_n|^{\beta}\leq \dfrac{\max g-\inf u^-+\dfrac{1}{n}}{b_n}.$$
Since $u^-$ is lower semi-continuous, so the infimum of $u^-$ is finite and then, $y_n-x_n\to 0$, concluding that $x_n\to y_0$ and so, $y_0\in \partial \Omega.$
We then get
$$u^-(y_0)=g(y_0)=\limsup_{n\to \infty}\, g(x_n)\geq \limsup_{n\to \infty} a_n\geq \limsup_{n\to \infty}  \phi_{x_n,a_n,b_n}(y_n)\geq \limsup_{n\to \infty} \bigg[u^-(y_n)-\dfrac{1}{n}\bigg].$$
Hence, $$\limsup_{n\to \infty} u^-(y_n)\leq u^-(y_0).$$

Assume now that $b_n$ has a bounded subsequence (we still call it $b_n$). Yet, we note that thanks to \eqref{phin} and the fact that $u^-$ is bounded from below and $a_n \leq \max g$, then $a_n$ is always bounded. So, we can extract subsequences, $(x_{n},a_n,b_n)\to (x_1,a_1,b_1)\in \partial \Omega\times \mathbb R\times [0,\infty).$
From \eqref{phin},
$$u^-(y_n)\leq a_n-b_n|y_n-x_n|^{\beta}+\dfrac{1}{n}.$$
So,
$$\limsup_{n\to \infty} u^-(y_n)\leq a_1-b_1|y_0-x_1|^{\beta}=\lim_{n\to \infty} a_n-b_n|y_0-x_n|^\beta=\lim_{n\to \infty} \phi_{x_n,a_n,b_n}(y_0)\leq u^-(y_0) .$$

We deduce that $u^-\in C(\overline{\Omega})$ and so we can use the stability result in Proposition \ref{stability} to conclude that $u^-$ is a subsolution to \eqref{Frac Inf Lap Eq} and $u^-=g$ on $\partial \Omega.$

Similarly, we define 
$$S^\prime=\bigg\{(x_0,a,b)\in \partial \Omega\times \mathbb R\times [0,\infty):  L[\phi_{x_0,a,-b}]\leq f(\phi_{x_0,a,-b})\,\,\,\,\text{in\,\, $\Omega$},\,\,{\phi_{x_0,a,-b}}\geq g\,\,\,\mbox{on}\,\,\,\partial\Omega\bigg\}.$$
Using the same approach, one can show that $u^+=\inf\limits_{(x_0,a,b)\in S^\prime} \phi_{x_0,a,-b}$ is in $C(\overline{\Omega})$ and is a supersolution to \eqref{Frac Inf Lap Eq} with $u^+=g$ on $\partial \Omega$.

Finally, we show that $u^-\leq u^+.$
Take $(x_0,a_0,b_0)\in S$ and $(x_0',a_0',b_0')\in S'$. Since $\beta<1$ then $\phi_{x_0,a_0,b_0}-\phi_{x_0',a_0',-b_0'}$ is a convex function over the compact domain $\overline{\Omega}$, then its global maximum is attained on $\partial \Omega$. But by definition of $S$ and $S'$, 
$\phi_{x_0,a_0,b_0}\leq g\leq \phi_{x_0',a_0',-b_0'}$ on $\partial\Omega$ concluding that $\phi_{x_0,a_0,b_0}\leq \phi_{x_0',a_0',-b_0'}$ on $\overline{\Omega}$. Yet, this is true for every $(x_0,a_0,b_0)\in S$ and $(x_0',a_0',b_0')\in S'$. Hence, $u^-\leq u^+$. $\qedhere$
\end{proof}

{
\begin{remark} \label{RK3.9}
Assume $g \geq 0$ on $\partial\Omega$. Then, for $(x_0,a_0,b_0)\in S'$, we have $\phi_{x_0,a_0,-b_0}\geq a_0 =\phi_{x_0,a_0,-b_0}(x_0)\geq g(x_0) \geq \min g$. Hence, we get that $u^+\geq 0$; this observation will be needed in Section \ref{sec:obstacle}.
\end{remark}
}
}
Now, we are ready to prove Theorem \ref{thm:Main 1}. First, we show the regularity of subsolutions and then prove the existence result.
\begin{proposition} \label{beta Holder regularity}
Assume $f$ is non-decreasing and continuous. Let \,$u$\, be a bounded viscosity subsolution of \eqref{Frac Inf Lap Eq}. Then,
$u$ is locally $\beta-$H\"olderian, for any $0<\beta <\alpha$. More precisely, we have
  $$[u]_{C^{0,\beta}(\omega)} \leq \max\bigg\{\frac{2||u||_\infty}{\dist(\omega,\partial\Omega)^\beta},\frac{[\mbox{diam}(\Omega)]^{\alpha-\beta}\,[f(-||u||_\infty)]_-}{1-\Psi(r_\star)}\bigg\},\qquad\mbox{for every}\,\,\,w \subset\subset \Omega,$$
 where using the notation of Proposition \ref{prop:alpha>beta}
  $$\Psi(r_\star)=\dfrac{r_\star^{\beta}-1}{(r_\star-1)^{\alpha}}<1.$$
  In addition, assume that $g \in C^{0,\beta}(\partial\Omega)$, if $u$ is a bounded viscosity solution of \eqref{Frac Inf Lap Eq} with $u=g$ on $\partial\Omega$, then $u$ is $\beta-$H\"olderian in $\overline{\Omega}$, and we have the following estimate:
     $$||u||_{C^{0,\beta}(\overline\Omega)} \leq C\bigg(\alpha,\,\beta,\,\mbox{diam}(\Omega),\,||g||_{C^{0,\beta}(\partial\Omega)},\,[ f(\pm||g||_\infty)]_\pm,\,[f(-||u||_\infty)]_-\bigg).$$
     \end{proposition}
  \begin{proof}
Fix $x_0 \in \omega \subset\subset \Omega$. Thanks to Proposition \ref{subsolution on domain minus point}, $u$ is a viscosity subsolution of \eqref{Frac Inf Lap Eq} on $\Omega\backslash\{x_0\}$. Now, we define 
    $$v(x)=u(x_0)+C|x-x_0|^\beta.$$
    Recalling Proposition \ref{prop:alpha>beta}, $v$ is a strict supersolution on $\Omega\backslash\{x_0\}$ since 
    $$-L[u(x_0)+C|x-x_0|^\beta]+f(u(x_0)+C|x-x_0|^\beta) \geq -C L[|x-x_0|^\beta] + f(-||u||_\infty)>0$$
as soon as $C\geq \frac{[\mbox{diam}(\Omega)]^{\alpha-\beta}\,[f(-||u||_\infty)]_-}{1-\Psi(r_\star)}$.
    Moreover, one has $v(x_0)=u(x_0)$, and for every $x\in \partial \Omega$, we have
    $$u(x) -v(x)=u(x)-u(x_0)-C|x-x_0|^\beta \leq 2||u||_\infty -C \,\dist(\omega,\partial\Omega)^\beta \leq 0 $$
    as soon as we choose the constant $C \geq 2||u||_\infty/\dist(\omega,\partial\Omega)^\beta$. Thanks to the comparison principle in Proposition \ref{comparison principle}, we infer that $u <v$ in $\Omega\backslash\{x_0\}$ and so, we have
    $$u(x) \leq u(x_0)+C|x-x_0|^\beta\qquad\mbox{for all}\,\,\,\,x \in \omega.$$
    Interchanging the role of $x_0$ and $x$, we get that
    $$|u(x)-u(x_0)| \leq C|x-x_0|^\beta.$$
    This shows that $u \in C^{0,\beta} (\omega)$.
    
Now, let us show the second statement. Assume $g \in C^{0,\beta}(\partial\Omega)$. Fix $x_0 \in \partial\Omega$. Then, we set the function 
    $$w^+(x)=g(x_0)+C\,|x-x_0|^\beta,$$
    where $C \geq [g]_{\beta,\partial\Omega}$.
    Since $g \in C^{0,\beta}(\partial\Omega)$ and $u=g$ on $\partial\Omega$, then we have for every $x \in \partial\Omega$ the following inequality:
    $$u(x)=g(x)\leq g(x_0)+C\,|x-x_0|^\beta=w^+(x).$$
    Moreover, one can show that $w^+$ is a strict supersolution provided that $C$ is large enough. Indeed,
\begin{align*}
-L[w^+]+f(w^+)&=-C\,L[\psi_{\beta,x_0}]+f(g(x_0)+C\,|x-x_0|^\beta)\geq C |x-x_0|^{\beta-\alpha}[1-\Psi(r_\star)] +f(-||g||_\infty) > 0
    \end{align*}
    provided that $$C\geq  \frac{\diam(\Omega)^{\alpha-\beta}\,[f(-||g||_\infty)]_-}{1-\Psi(r_\star)}.$$
    Thanks again to the comparison principle in Proposition \ref{comparison principle}, we get that \,$u <w^+$\, in $\Omega$. Therefore, one has
    $$u(x) \leq g(x_0)+C\,|x-x_0|^\beta,\qquad\mbox{for all}\,\,\,x \in \overline{\Omega}.$$
    In the same way, we set  $$w^-(x)=g(x_0)-C\,|x-x_0|^\beta,$$
    where $C>0$ is a large constant that we will choose later. We claim that $w^-$ is a strict subsolution. In fact, we have the following:
    $$ -L[w^-]+f(w^-)=C\,L[\psi_{\beta,x_0}]+f(g(x_0)-C\,|x-x_0|^\beta)
    \leq C |x-x_0|^{\beta-\alpha}[\Psi(r_\star) -1] + f(||g||_\infty) < 0$$
    provided that $$C\geq  \frac{\diam(\Omega)^{\alpha-\beta}\,[f(||g||_\infty)]_+}{1-\Psi(r_\star)}.$$
    \\
    On the other hand, assuming that the constant $C \geq [g]_{\beta,\partial\Omega}$ then we get thanks to the H\"older regularity of $g$ that 
     $$u(x)=g(x) \geq w^-(x)=g(x_0)-C\,|x-x_0|^\beta,\qquad\mbox{for all}\,\,\,x \in \partial{\Omega}.$$
     Hence,
     $$u(x) \geq w^-(x),\qquad\mbox{for all}\,\,\,x \in \overline{\Omega}.$$
     Consequently,
     $$w^-(x) \leq u(x) \leq w^+(x),\qquad\mbox{for all}\,\,\,x \in \overline{\Omega}.$$
     Yet, $u(x_0)=g(x_0)$. Then, 
      $$|u(x)-u(x_0)| \leq C\,|x-x_0|^\beta,\qquad\mbox{for all}\,\,\,x \in \overline{\Omega}. $$
    \end{proof}
    {{
    \begin{proposition}\label{alpha Holder reg}
    Assume \,$0<\alpha <1$, $f \geq 0$ and $f$ is non-decreasing and continuous, then any bounded viscosity subsolution $u$ of \eqref{Frac Inf Lap Eq} is locally $\alpha-$H\"olderian. In addition, we have
    $$[u]_{C^{0,\alpha}(\omega)} \leq \frac{2||u||_\infty}{\dist(\omega,\partial\Omega)^\alpha},\qquad\mbox{for every}\,\,\,w \subset\subset \Omega.$$
    \end{proposition}}}
    \begin{proof}
    {{
 We will follow the same argument as in Proposition \ref{beta Holder regularity}.
 Let $\omega \subset \subset \Omega$ and fix $x_0 \in \omega$. Set   $$v(x)=u(x_0)+C|x-x_0|^\alpha.$$
    Thanks to Propositon \ref{alpha=beta} and the fact that $f \geq 0$, one has
    $$-L[u(x_0)+C|x-x_0|^\alpha]+f(u(x_0)+C|x-x_0|^\alpha) \geq -C L[|x-x_0|^\alpha]>0.$$
    Hence, $v$ is a strict supersolution in $\Omega\backslash\{x_0\}$. 
    On the other hand, $v(x_0)=u(x_0)$. For every $x\in \partial \Omega$, we also have
    $$u(x) -v(x)=u(x)-u(x_0)-C|x-x_0|^\alpha \leq 2||u||_\infty -C \,\dist(\omega,\partial\Omega)^\alpha  \leq 0$$
    provided that $C$ is large enough. By the comparison principle \eqref{comparison principle}, this yields that $u <v$ in $\Omega\backslash\{x_0\}$. Hence, we get
    $$u(x) \leq u(x_0)+C|x-x_0|^\alpha\qquad\mbox{for all}\,\,\,\,x \in \omega. \qedhere$$
 }}
 \end{proof}
Now, we are ready to prove our existence result.
\begin{theorem} \label{existence of a solution}
Under the assumptions that $f$ is non-decreasing and continuous, and {{$g$ is continuous}} on $\partial\Omega$, there exists a viscosity solution $u$ to
    Problem \eqref{Frac Inf Lap Eq}.
\end{theorem}
\begin{proof}
{{From Lemma \ref{Existence of subsupersolutions}, we know that there exist a subsolution $u^-$ and a supersolution $u^+$ to \eqref{Frac Inf Lap Eq} with $u^\pm \in C(\overline{\Omega})$, $u^- \leq u^+$ on \,$\overline{\Omega}$\, and \,$u^+=u^-=g$\, on \,$\partial\Omega$.}}
Set 
$$S=\bigg\{w \in C(\overline{\Omega}) \,\,\,\mbox{is a subsolution}\,:\,\,u^- \leq w \leq u^+\,\,\,\mbox{on}\,\,\,\overline{\Omega}\bigg\}.$$
First, we note that \,$S \neq \emptyset$\, since the subsolution $u^-$ constructed in Lemma \ref{Existence of subsupersolutions} belongs to $S$. Then, we define the function 
$$u=\sup_{w \in S} w.$$
For any $w \in S$, we clearly have  $$||w||_\infty \leq \max\{||u^+||_\infty,||u^-||_\infty\}:=\Lambda.$$
On the other hand, by Proposition \ref{beta Holder regularity} , one has
 $$[w]_{C^{0,\beta}(\omega)} \leq \max\bigg\{\frac{2\Lambda}{\dist(\omega,\partial\Omega)^\beta},\frac{[\mbox{diam}(\Omega)]^{\alpha-\beta}\,[f(-\Lambda)]_-}{1-\Psi(r_\star)}\bigg\},\qquad\mbox{for every}\,\,\,w \subset\subset \Omega.$$\\
{{Hence, we infer that $u$ is locally $\beta-$H\"older as it is the supremum of uniformly locally $\beta-$H\"older functions. So in particular, $u$ is continuous in $\Omega$.}}
{{Yet, $u$ is continuous on $\partial\Omega$ since $u^-\leq u \leq u^+$ on $\overline{\Omega}$, $u^\pm \in C(\overline{\Omega})$ and, $u^+=u^-=g$ on $\partial\Omega$.}} From Proposition \ref{stability}, this implies that $u$ is a subsolution of \eqref{Frac Inf Lap Eq} with $u=g$ on $\partial\Omega$. 


Let us show that $u$ is also a supersolution, so that it will be a viscosity solution. Assume that this is not the case, i.e. there is a point $x_0 \in \Omega$ and a function $\varphi \in C^1(\Omega) \cap C(\overline{\Omega})$ such that $u \geq \varphi$ on $\overline{\Omega}$ with $u(x_0)=\varphi(x_0)$ and 
\begin{equation} \label{3.2.1}
-L[\varphi](x_0) +f(\varphi(x_0))<0.
\end{equation}
We recall that we may assume $x_0$ to be the unique minimum of $u-\varphi$. Indeed, for $\delta>0$ small enough, set $\varphi_{\delta}(x)=\varphi(x)-\delta |x-x_0|^2$. Hence, we clearly have $\varphi_\delta \leq \varphi \leq u$ on $\overline{\Omega}$ with $\varphi_\delta(x_0)=\varphi(x_0)=u(x_0)$. Moreover, by Lemma \ref{lem:delta perturbation}, we have 
$$-L[\varphi_\delta](x_0) +f(\varphi_\delta(x_0)) \leq -L[\varphi](x_0)+C\delta +f(\varphi(x_0)) <0,$$
as soon as $\delta>0$ is sufficiently small. 

Now, we claim that $u(x_0)<u^+(x_0)$. Suppose it is not the case, i.e. we have \,$u(x_0)=u^+(x_0)$. Hence, we infer that $\varphi \leq u \leq u^+$ on $\overline{\Omega}$ with $u^+(x_0)=\varphi(x_0)$, and having $u^+$ is a viscosity supersolution, then 
$$-L[\varphi](x_0) +f(\varphi(x_0)) \geq 0,$$
which is a contradiction. 

Since $u^+$, $\varphi$ are continuous on $\Omega$, $\varphi(x_0)=u(x_0)<u^+(x_0)$, $\varphi \leq u \leq u^+$ and $x_0$ is the unique minimum of $u-\varphi$, then there will be a small constant $\zeta_0>0$ such that $\varphi+\zeta \leq u^+$ on $\overline{\Omega}$, for all $0<\zeta \leq \zeta_0$.
We set 
$$u_\zeta=\max\{u\,,\,\varphi + \zeta\}.$$
We shall prove that $u_\zeta \in S$, for $\zeta>0$ small enough. In this case, since $u \geq u_\zeta$ on $\overline{\Omega}$, one has in particular at $x=x_0$ that 
$$ u(x_0) \geq u_\zeta(x_0)=\varphi(x_0)+\zeta=u(x_0)+\zeta,$$
which is clearly a contradiction as $\zeta>0$. Hence, it remains to prove the claim that $u_\zeta \in S$. First, it is clear that $u^- \leq u \leq  u_\zeta \leq u^+$. Let us show that $u_\zeta$ is a subsolution. Assume it is not the case, so there exists a point  $x_\zeta \in \Omega$ and a function $\varphi_\zeta \in C^1(\Omega) \cap C(\overline{\Omega})$ such that \,$u_\zeta \leq \varphi_\zeta$\, on \,$\overline{\Omega}$\, and \,$u_\zeta(x_\zeta)=\varphi_\zeta(x_\zeta)$ with 
\begin{equation} \label{3.2.2}
 -L[\varphi_\zeta](x_\zeta)+f(\varphi_\zeta(x_\zeta))>0.
 \end{equation}
Here, we have two possibilities: either \,$u_\zeta(x_\zeta)=u(x_\zeta)$ or $u_\zeta(x_\zeta)=\varphi(x_\zeta)+\zeta$. If \,$u_\zeta(x_\zeta)=u(x_\zeta)$ for some $\zeta$, then we have $u(x_\zeta)=\varphi_\zeta(x_\zeta)$ and $u \leq u_\zeta \leq \varphi_\zeta$. But $u$ is a subsolution, then we must have  
$$ -L[\varphi_\zeta](x_\zeta)+f(\varphi_\zeta(x_\zeta)) \leq 0,$$
which is a contradiction. 

The remaining case is when $u_\zeta(x_\zeta)=\varphi(x_\zeta)+\zeta$ for all $\zeta$ small, so $\varphi_\zeta(x_\zeta)=\varphi(x_\zeta) +\zeta$. Since $\varphi+\zeta \leq u_\zeta \leq \varphi_\zeta$, then one has
$$ \varphi \leq \varphi_\zeta -\zeta \qquad \mbox{on}\,\,\,\,\,\overline{\Omega}.$$
Hence, we have 
$$L[\varphi](x_\zeta) \leq L[\varphi_\zeta-\zeta](x_\zeta)= L[\varphi_\zeta](x_\zeta).$$
In particular, we get that
$$- L[\varphi_\zeta](x_\zeta) +f(\varphi(x_\zeta))\leq -L[\varphi](x_\zeta) + f(\varphi(x_\zeta)).$$
Consequently, 
\begin{equation} \label{3.2.3}
[- L[\varphi_\zeta](x_\zeta)+f(\varphi_\zeta(x_\zeta)) ]+[f(\varphi(x_\zeta))-f(\varphi_\zeta(x_\zeta))]\leq -L[\varphi](x_\zeta) + f(\varphi(x_\zeta)).
\end{equation}\\
Recalling \eqref{3.2.2}, \eqref{3.2.3} yields to
\begin{equation} \label{3.2.4}
f(\varphi(x_\zeta))-f(\varphi_\zeta(x_\zeta))\leq -L[\varphi](x_\zeta) + f(\varphi(x_\zeta)).
\end{equation}
However, we claim that the sequence of points $x_\zeta$ converges to $x_0$. Otherwise, it means that up to a subsequence $x_\zeta \rightarrow x^\star \neq x_0$. But, we have 
$$u(x_\zeta) \leq u_\zeta(x_\zeta) = \varphi(x_\zeta) +\zeta,\,\,\,\mbox{for all}\,\,\,\zeta.$$
Letting $\zeta \to 0^+$, we infer that $u(x^\star) \leq \varphi(x^\star)$. Hence, $u(x^\star)=\varphi(x^\star)$ and $x^\star$ is a minimum point of $u- \varphi$. Yet, $x_0$ is the unique minimum point for $u- \varphi$ and so, $x^\star=x_0$. Yet, this is also a contradiction. So, our claim is proved. 

Passing to the limit in \eqref{3.2.4}, we get
$$0 \leq  -L[\varphi](x_0) + f(\varphi(x_0)).$$
But, this contradicts \eqref{3.2.1}. Hence, this concludes the proof that $u_\zeta \in S$. 
\end{proof}

We finish this section by the following observation that we will use when dealing with the obstacle problem.
\begin{remark} \rm\label{Remark 2}
   Assume the boundary datum $g \geq 0$ on $\partial\Omega$ (but, $g$ is not identically zero). Recalling Remark \ref{RK3.9},
   the supersolution $u^+$ constructed in Lemma \ref{Existence of subsupersolutions} is nonnegative. However, the subsolution $u^-$ defined in Lemma \ref{Existence of subsupersolutions} is not necessarily nonnegative.
   
   Now, assume $f= 0$ on ${{(-\infty,0]}}$. Then, there will always be a nonnegative subsolution $u^-$ such that $u^- \leq u^+$ on $\overline \Omega$ and $u^+=u^-=g$ on $\partial\Omega$. In fact, it is easy to see that $w^\star:=\max\{u^-,0\}$ is also a subsolution with $w^\star=g$ on $\partial\Omega$. Moreover, we have \,$w^\star \leq u^+$. Hence, { using the setting of Proposition \ref{existence of a solution}}, $w^\star \in S$. From the definition of the Perron's solution $u$, this yields that 
$$u=\sup_{w \in S} w \geq w^\star \geq 0.$$
Then, $u \geq 0$ on $\overline{\Omega}$.

Finally, assume that there is a point $x_0 \in \Omega$ such that $u(x_0)=0$. Since $u \neq 0$ and $u \in C(\overline{\Omega})$, then there is a point $x^\star \neq x_0 \in \Omega$ such that $u >\frac{u(x^\star)}{2}>0$ on $B(x^\star,\varepsilon)$, where $\varepsilon>0$ is small enough. Now, let $\varphi  \in C^1(\Omega) \cap C(\overline{\Omega})$ be such that $\varphi \neq 0$, $\mbox{supp}(\varphi) \subset B(x^\star,\varepsilon)$ and $0\leq \varphi \leq \frac{u(x^\star)}{4}$. In particular, we have $u \geq \varphi$ and $u(x_0)=\varphi(x_0)=0$. Therefore, we must have 
$$0<L[\varphi](x_0) \leq f(\varphi(x_0))= 0,$$
which is a contradiction.

Notice also that when $f=0$ on ${{(-\infty,0]}}$, we obtain from Proposition \ref{beta Holder regularity} the following uniform (does not depend on the solution $u$) estimate:
 $$||u||_{C^{0,\beta}(\overline\Omega)} \leq C\bigg(\alpha,\,\beta,\,\mbox{diam}(\Omega),\,||g||_{C^{0,\beta}(\partial\Omega)},\, f(||g||_\infty)\bigg).$$
\end{remark}
\section{Obstacle problem}\label{sec:obstacle}
In this section, we assume that $f:[0,\infty) \mapsto \mathbb{R}$ is continuous, nonnegative and non-decreasing, and the boundary datum $g$ is nonnegative and continuous on $\partial\Omega$. We prove Theorem \ref{Th1.2} by showing that there exists a nonnegative function $u$ that is 
solution to the following obstacle problem
\begin{equation}\label{N system with obstacle}
\begin{cases} L[u]=f(u)\qquad &\text{in\,\, $\{u>0\}$},\\
u=g\qquad &\text{on \,\,$\partial \Omega$}.\end{cases}
\end{equation}
\begin{proof}[Proof of Theorem \ref{Th1.2}]
In the case when $f(0)=0$, we extend $f$ by 0 on ${{(-\infty,0)}}$. Then, this extension (we still denote it by $f$) is continuous and non-decreasing on $\mathbb{R}$ and so, by Proposition \ref{existence of a solution}, Problem \eqref{Frac Inf Lap Eq} has a solution $u$. Thanks to Remark \ref{Remark 2}, $u>0$. Hence, $u$ solves Problem \eqref{N system with obstacle}.  

Now, we consider the case when $f(0)>0$. Assuming $\alpha<1$.
Let $f_\varepsilon$ be a sequence of non-decreasing continuous functions such that $f_\varepsilon=0$ on ${{(-\infty,0]}}$ and $f_\varepsilon=f$ on $[\varepsilon,+\infty)$.
For every $\varepsilon >0$, by Proposition \ref{existence of a solution}, we know that there exists a solution $u_\varepsilon$ to Problem \eqref{Frac Inf Lap Eq} with \,$u_\varepsilon=g$\, on \,$\partial\Omega$. Recalling Remark \ref{Remark 2}, we may assume that $u_\varepsilon>0$ on $\Omega$. In addition, by Proposition \ref{beta Holder regularity}, we have that
 $$||u_\varepsilon||_{C^{0,\beta}(\overline\Omega)}  \leq C\bigg(\alpha,\,\beta,\,\mbox{diam}(\Omega),\,||g||_{C^{0,\beta}(\partial\Omega)},\,f(||g||_\infty)\bigg)
 ,$$
 for $\varepsilon>0$ small enough. 
 
Hence, $(u_\varepsilon)_\varepsilon$ is bounded in $C^{0,\beta}(\overline\Omega)$. Therefore, up to a subsequence, $u_\varepsilon \rightarrow u$ 
 uniformly in \,$C^{0,\beta}(\overline\Omega)$\, and \,$u \geq 0$\, on \,$\overline{\Omega}$.

We will show  that $u$ is a viscosity subsolution to \eqref{N system with obstacle} (the fact that $u$ is a supersolution can be treated similarly). Therefore, $u$ will be a viscosity solution for Problem \eqref{N system with obstacle} with boundary datum $u=g$. Assume by contradiction that $u$ is not a subsolution then there exists $x_0\in\{u>0\}$ and a function $\varphi \in C(\overline{\Omega}) \cap C^1(\Omega)$ such that $u \leq \varphi$ on $\overline{\Omega}$ with $u(x_0)=\varphi(x_0)$ but
$$-L[\varphi](x_0) +f(\varphi(x_0)) > 0.$$
 Thanks to the uniform convergence of $u_\varepsilon$ to $u$, one can find a sequence $\varphi_\varepsilon$ converging uniformly to $\varphi$ such that $\varphi_\varepsilon \in C(\overline{\Omega}) \cap C^1(\Omega)$, \,$u_\varepsilon
\leq \varphi_\varepsilon$\, on \,$\overline{\Omega}$\, and $u_\varepsilon(x_\varepsilon)=\varphi_\varepsilon(x_\varepsilon)$, where $x_\varepsilon \to x_0$ when $\varepsilon \to 0$. Since $u_\varepsilon$ is a viscosity solution, then 
\begin{equation} \label{5.0.1}
-L[\varphi_\varepsilon](x_\varepsilon) +f_\varepsilon(\varphi_\varepsilon(x_\varepsilon))\leq 0.
\end{equation}
Yet, 
$$L[\varphi_\varepsilon](x_\varepsilon)=\sup_{y\in \bar \Omega,\,y\neq x_\varepsilon} \dfrac{\varphi_\varepsilon(y)-\varphi_\varepsilon(x_\varepsilon)}{|y-x_\varepsilon|^{\alpha}}+\inf_{y\in \bar \Omega,\,y\neq x_\varepsilon} \dfrac{\varphi_\varepsilon(y)-\varphi_\varepsilon(x_\varepsilon)}{|y-x_\varepsilon|^{\alpha}}.$$
We claim that 
$$|L[\varphi_\varepsilon](x_\varepsilon) -L[\varphi](x_\varepsilon)|\leq C ||\varphi_\varepsilon -\varphi||_\infty^{1-\alpha}.$$
We will show this inequality for $L^+$ (the proof for $L^-$ will be similar and so, it will be omitted). First, one has
$$\dfrac{\varphi_\varepsilon(y)-\varphi_\varepsilon(x_\varepsilon)}{|y-x_\varepsilon|^{\alpha}}=\dfrac{\varphi(y)-\varphi(x_\varepsilon)}{|y-x_\varepsilon|^{\alpha}} +\dfrac{\varphi_\varepsilon(y)-\varphi(y) +\varphi(x_\varepsilon)-\varphi_\varepsilon(x_\varepsilon)}{|y-x_\varepsilon|^{\alpha}}.$$
But, it is clear that
$$ \dfrac{\varphi_\varepsilon(y)-\varphi(y) +\varphi(x_\varepsilon)-\varphi_\varepsilon(x_\varepsilon)}{|y-x_\varepsilon|^{\alpha}} \leq 2\, 
\min\bigg\{\dfrac{||\varphi_\varepsilon-\varphi||_\infty}{|y-x_\varepsilon|^{\alpha}},C\,|y-x_\varepsilon|^{1-\alpha}\bigg\},
$$
where $C<\infty$ is a uniform constant such that $\mbox{Lip}(\varphi_\varepsilon),\,\mbox{Lip}(\varphi) \leq C$ on $\overline{B(x_0,\delta)}$, for $\delta>0$ small enough.
Then, we get that
$$ \dfrac{\varphi_\varepsilon(y)-\varphi(y) +\varphi(x_\varepsilon)-\varphi_\varepsilon(x_\varepsilon)}{|y-x_\varepsilon|^{\alpha}} \leq C\, 
\min\bigg\{\dfrac{||\varphi_\varepsilon-\varphi||_\infty}{|y-x_\varepsilon|^{\alpha}},\,|y-x_\varepsilon|^{1-\alpha}\bigg\}\leq C ||\varphi_\varepsilon-\varphi||_\infty^{1-\alpha}.
$$
On the other hand, it is clear that $\varphi_\varepsilon(x_\varepsilon) \rightarrow \varphi(x_0)>0$ and so, $f_\varepsilon(\varphi_\varepsilon(x_\varepsilon))=f(\varphi_\varepsilon(x_\varepsilon)) \rightarrow f(\varphi(x_0))$. Hence, thanks to Lemma \ref{lem:continuity} and passing to the limit when $\varepsilon \to 0$ in \eqref{5.0.1}, we infer that 
$$-L[\varphi](x_0) +f(\varphi(x_0))\leq 0,$$
which contradicts our main assumption.

{{Thanks to Proposition \ref{alpha Holder reg} and since $f_\varepsilon \geq 0$, then $u_{\epsilon}$ are uniformly (in $\epsilon$) locally $\alpha-$H\"older continuous. Moreover, one has
 $$[u_\varepsilon]_{C^{0,\alpha}(\omega)} \leq \frac{2||u_\varepsilon||_\infty}{\dist(\omega,\partial\Omega)^\alpha} \leq \frac{C}{\dist(\omega,\partial\Omega)^\alpha},\qquad\mbox{for every}\,\,\,w \subset\subset \Omega.$$
 Then, letting $\epsilon\to 0$, we infer that the limit function $u$ is locally $\alpha-$H\"older continuous. }}
\end{proof}
{{
\begin{remark}
    We note that it is not clear whether we can  prove global $\alpha-$H\"older regularity on the solution $u$ in the previous theorem or not, unless $f=0$. Indeed, recalling the proof of Proposition \ref{beta Holder regularity}, the key idea is to show that, at each boundary point $x_0 \in \partial\Omega$, the function $u_\varepsilon$  admits two barrier functions from above and below of the form
    $w^\pm(x)=g(x_0)\pm C\,|x-x_0|^\alpha$. The existence of a barrier function $w^+$ from above is not an issue (see the proof of Proposition \ref{beta Holder regularity}). However, the difficulty arises when showing a barrier function $w^-$ from below since now (i.e. when $\beta=\alpha$) we lose the upper bound $L[|x-x_0|^\beta] \leq -c<0$ (which is true if $\beta<\alpha$; see Proposition \ref{prop:alpha>beta}). In fact, set 
    $$w^-(x)=g(x_0)-C\,|x-x_0|^\alpha.$$
    We have 
    $$ -L[w^-]+f(w^-)=C\,L[|x-x_0|^\alpha]+f(g(x_0)-C\,|x-x_0|^\alpha).$$
    By Proposition \ref{alpha=beta}, we get
    $$ -L[w^-]+f(w^-) \leq 
    -1+\dfrac{ \left(\dfrac{\diam(\Omega)}{|x-x_0|}\right)^\alpha-1}{\left(\dfrac{\diam (\Omega)}{|x-x_0|}-1\right)^{\alpha}} +f(g(x_0)-C\,|x-x_0|^\alpha).
    $$
    But, the right hand side in the inequality above is not strictly negative as soon as $x$ is close enough to $x_0$ and $f(g(x_0))>0$. Hence, it is not clear whether $w^-$ is a strict subsolution to \eqref{Frac Inf Lap Eq} or not.
\end{remark}
\begin{remark}
    We note that the solution $u$ of the obstacle problem \eqref{N system with obstacle} is not necessarily strictly positive. Thus, an interesting question will be to study the properties of the free boundary $\partial\{u>0\}$, 
    this remains open. 
    
    On the other hand, regularity results of the solution on the free boundary for the same problem \eqref{N system with obstacle} but with infinity (instead of fractional) Laplacian has been analyzed for example in \cite{Diehl} through a scaling argument that could not be adapted to our  case since the operator $L$ is non-local. Finally, we mention that for the infinity Laplacian we obtain
higher regularity on the solution across the free boundary. However, it is not clear if this will be also the case when dealing with fractional infinity Laplacian. 
\end{remark}
}}
{
\section*{Acknowledgement} The authors would like to thank the referee for all their comments and in particular for their suggestions to refine our results under weaker regularity assumptions on the boundary datum $g$.
}

\end{document}